\documentclass{article}
\usepackage{amsmath, amsthm, amsfonts, amssymb, amscd, stmaryrd,xcolor,enumerate, enumitem, relsize}

\usepackage{graphicx,filecontents}
\usepackage{tikz}

\usepackage[
    backend=biber,
    style=numeric,
    maxbibnames=99
  ]{biblatex}
\renewbibmacro{in:}{%
  \ifentrytype{article}
    {}
    {\bibstring{in}%
     \printunit{\intitlepunct}}}
\DeclareNameAlias{default}{family-given}

\theoremstyle{definition}

\theoremstyle{plain}
\newtheorem{Thm}{Theorem}
\newtheorem{Lem}[Thm]{Lemma}
\newtheorem{Cor}[Thm]{Corollary}

\newtheorem{Prop}[Thm]{Proposition}

\newtheorem{Examp}[Thm]{Example}

\numberwithin{equation}{section}	
\numberwithin{figure}{Prob}
\numberwithin{table}{Prob}
\numberwithin{Thm}{section}

\newcommand{\Z}{\mathbb{Z}}

\newcommand{\C}{\mathbb{C}}

\newcommand{\normal}{\trianglelefteq}

\DeclareMathOperator*{\Aut}{Aut}

\renewcommand{\S}{\mathfrak{S}}

\newcommand{\SL}{\text{SL}}

\title{Almost Commutative Terwilliger Algebras II: Strong Gelfand Pairs}
\author{Nicholas L. Bastian, Stephen P. Humphries}

\date{\today}

\begin{document}

\maketitle

\begin{abstract}
Terwilliger algebras are a subalgebra of a matrix algebra constructed from an association scheme. In 2010, Tanaka defined what it means for a Terwilliger algebra to be almost commutative and gave five equivalent conditions for a Terwilliger algebra to be almost commutative. In this paper we look at Terwilliger algebras coming from strong Gelfand pairs $(G,H)$ for a finite group $G$. From such a pair, one can create a Terwilliger algebra using the Schur ring of $H-$classes of elements of $G$. We determine all strong Gelfand pairs that give an Almost Commutative Terwilliger algebra.      \end{abstract}



\section{Introduction}

Terwilliger algebras were originally developed in the 1990's by Paul Terwilliger to study commutative association schemes. In particular, he looked at P-Polynomial and Q-Polynomial association schemes. In \cite{Terwilliger1,Terwilliger2,Terwilliger3}, where Terwilliger algebras were introduced, Terwilliger finds a combinatorial characterization of thin P-Polynomial and Q-Polynomial association schemes. 
 
Tanaka studied a situation in which the Terwilliger algebra has a particular Wedderburn structure (see Theorem \ref{thm:tanaka}) called almost commutative. One of Tanaka's results is that when a Terwilliger algebra is almost commutative, the corresponding association scheme is a wreath product. In \cite{Bastian2}, we determined when a group association scheme results in an almost commutative Terwilliger algebra (see Theorem \ref{thm:acclassify}). This corresponds to the strong Gelfand pair $(G,G)$. A natural generalization of this is to consider any strong Gelfand pair $(G,H)$, since then the $H-$classes $g^H$ determine a commutative association scheme. 

 In this paper all groups $G$ are finite. We let $\mathcal{G}(G)=Z(\C[G])=\C[G]^G$ denote the group association scheme and we let $T(G)$ denote the Terwilliger algebra for $\mathcal{G}(G)$ with base point $e$. We let $\mathcal{K}_n$ denote the trivial association scheme for a set of size $n$. Given a strong Gelfand pair $(G,H)$, we let $T(G,\C[G]^H)$ denote the Terwilliger algebra resulting from the $H-$class Schur ring $\C[G]^H$. The main result of this paper, which uses results of \cite{Bastian2}, is 

 \begin{Thm}
    Let $G$ be a finite group and $H\lneq G$. Then $(G,H)$ is a strong Gelfand pair and $T(G,\C[G]^H)$ is almost commutative if and only if
    \begin{enumerate}
        \item  $G$ is an abelian group and $H$ is any proper subgroup of $G$. In this case, $\dim T(G,\C[G]^H)=|G|^2$.
        \item $G$ is a Frobenius group with Frobenius kernel $H$ and cyclic complement $\Z_k$ such that $H$ is either an abelian group or a Camina $p-$group. In this case, the corresponding association scheme is $\mathcal{G}(H)\wr \mathcal{G}(G/H)$. Additionally, if $H$ has $m$ conjugacy classes then 
        \[\dim T(G,\C[G]^H)=\dim T(H)+(k-1)(k-2+3m).\]
    \end{enumerate}
\end{Thm}

The situation where $(G,H)$ is a strong Gelfand pair and $H$ is abelian has been considered in \cite{Humphries2}, where it is referred to as an extra strong Gelfand pair.

This paper is structured as follows. Section $2$ will discuss some preliminary definitions and results. In Section $3$ we will determine a classification of Almost Commutative Terwilliger algebras for strong Gelfand pairs (Theorem \ref{thm:sgpclassify}). Then we determine the dimension of the corresponding Terwilliger algebra (Proposition \ref{prop:sgpdim}). In Section $4$ we determine how to write an association scheme coming from a strong Gelfand pair as a wreath product (Corollary \ref{cor:sgpwreath}). One can then find the Wedderburn components of each of the Terwilliger algebras using results from \cite{WreathProduct1}.

Some notation that will be used throughout includes: given elements $g,h\in G$, we let $g^h=h^{-1}gh$ and $[g,h]=g^{-1}h^{-1}gh$. Furthermore for $H\leq G$ we let $g^H=\{h^{-1}gh\colon h\in H\}$. Also for $H,K\leq G$, $[K,H]=\langle [k,h]\colon k\in K,h\in H\rangle$.\\ 
\textbf{Acknowledgment:} All computations made in the preparation of this paper were performed using Magma \cite{Magma}.

\section{Preliminaries}

For definitions of association schemes, Schur rings, and Terwilliger algebras we refer to \cite{Bastian2}. We remind the reader of the following definition and results also found in \cite{Bastian2}: A Terwilliger algebra $T(x)$ is \emph{almost commutative} (AC) if every non-primary irreducible $T(x)-$module is $1-$ dimensional. There is a classification of such Terwilliger algebras when the underlying association scheme is commutative:

\begin{Thm}[Tanaka, \cite{Tanaka}]\label{thm:tanaka}  Let $\mathcal{A}=(\Omega,\{R_i\}_{0\leq i\leq d})$ be a commutative association scheme. Let $T(x)$ be the Terwilliger algebra of $\mathcal{A}$ for some $x\in \Omega$. The following are equivalent:
    \begin{enumerate}
        \item Every non-primary irreducible $T(x)-$module is $1-$ dimensional for some $x\in \Omega$.
        \item Every non-primary irreducible $T(x)-$module is $1-$ dimensional for all $x\in \Omega$.
        \item The $p_{ij}^h$ satisfy: for all distinct $h,i$ there is exactly one $j$ such that $p_{ij}^h\neq 0$ $(0\leq h,i,j\leq d)$.
        \item The $q_{ij}^h$ satisfy: for all distinct $h,i$ there is exactly one $j$ such that $q_{ij}^h\neq 0$ $(0\leq h,i,j\leq d)$.
        \item $\mathcal{A}$ is a wreath product of association schemes $\mathcal{A}_1,\mathcal{A}_2,\cdots, \mathcal{A}_n$ where each $\mathcal{A}_i$ is either a $1-$class association scheme or the group scheme of a finite abelian group.
    \end{enumerate} 
Moreover, a Terwilliger algebra satisfying these equivalent conditions is triply regular.
\end{Thm}

 We now give a sixth equivalent condition for a Terwilliger algebra to be AC in the case when the association scheme comes from a commutative Schur ring.

 \begin{Thm}\label{thm:sequiv}[Theorem 2.2, \cite{Bastian2}]
    Let $G$ be a finite group. Let $\S$ be any commutative Schur ring for $G$ with principal sets $P_i$, $0\leq i\leq d$. Let $T(G,\S)$ be the Terwilliger algebra of $G$ with respect to $\S$ and with base point $e$. Then $T(G,\S)$ is AC if and only if $\S$ has the following property: for all $x,y\in G$ if $x\in P_i$, $y\in P_j$, $xy\in P_h$, and $P_i\neq P_j^*$ then $P_iP_j=P_h$. 
\end{Thm}

A pair $(G,H)$, $H\leq G$ is called a \emph{strong Gelfand pair} (SGP) if the sets $P_i=g_i^H=\{g_i^h\colon h\in H\}$ form a commutative Schur ring. We let $\C[G]^H$ denote the resulting commutative Schur ring.

We note that strong Gelfand pairs have an equivalent definition using character theory. See \cite{Harmonic,Karlof}, for a proof of the equivalence. More information on SGPs can be found in \cite{Aizenbud2,Aizenbud1,Humphries2, Humphries1,Humphries3}. We shall need:

\begin{Lem}\label{lem:sgpmod}[Lemma 3.2, \cite{Humphries1}]
    Suppose $N\leq H\leq G$ with $N\normal G$. If $(G,H)$ is a SGP, then $(G/N,H/N)$ is a SGP. 
\end{Lem}

A group $G$ is a \emph{Frobenius group} if there is a non-trivial proper subgroup $N\normal G$ such that if $e\neq x\in N$ then $C_G(x)\leq N$. We call $N$ the \emph{Frobenius kernel} of $G$. Several equivalent conditions are:

\begin{Prop}[\cite{Isaacs}, p. 121]\label{prop:frobeq}
    Let $N\normal G$, $H\leq G$ with $NH=G$ and $N\cap H=\{e\}$. Then the following are equivalent.
    \begin{enumerate}
        \item $C_G(n)\leq N$ for all $n\in N\setminus \{e\}$.
        \item $C_H(n)=\{e\}$ for all $n\in N\setminus \{e\}$.
        \item $C_G(h)\leq H$ for all $h\in H\setminus \{e\}$.
        \item Every $x\in G\setminus N$ is conjugate to an element of $H$.
        \item If $h\in H\setminus \{e\}$, then $h$ is conjugate to every element of $hN$.
        \item $G$ is a Frobenius group with Frobenius kernel $N$.
    \end{enumerate}
\end{Prop}

Here $H$ is called the \emph{Frobenius complement}. We will need:

\begin{Thm}[Thompson, \cite{Thompson}]\label{thm:frobnil}
    Let $G$ be a Frobenius group and let $N$ be its Frobenius kernel. Then $N$ is nilpotent.
\end{Thm}

A group $G$ is a \emph{Camina group} if every conjugacy class of $G$ outside of $G'$ is a coset of $G'$. Camina groups \cite{Camina,Lewis1} are a generalization of Frobenius groups and extra special groups. We note:

\begin{Thm}[Dark and Scoppola, \cite{Dark1}]\label{thm:camclassify} Let $G$ be a Camina group. Then one of the following is true:
        \begin{itemize}[leftmargin=*]
            \item $G$ is a Frobenius group whose Frobenius complement is cyclic.
            \item $G$ is a Frobenius group whose Frobenius complement is $Q_8$, the quaternion group of order $8$.
            \item $G$ is a $p-$group of nilpotency class $2$ or $3$ for some prime $p$.
        \end{itemize}
        In particular, Camina groups are solvable.
\end{Thm}

\begin{Prop}[Macdonald, Theorem 5.2(i), \cite{MacDonald}]\label{lem:pcamclass}
    Let $G$ be a Camina $p-$group. Let $\gamma_2(G)=[G,G]$ and $\gamma_3(G)=[\gamma_2(G),G]=Z(G)$ be the second and third terms in the lower central series of $G$. Then
    \[\begin{array}{cc}
        x^G=x\gamma_2(G) & \text{ if }x\in G\setminus \gamma_2(G),   \\
        x^G=x\gamma_3(G)& \text{ if }x\in \gamma_2(G)\setminus \gamma_3(G),\\
        x^G=\{x\}\ \ \ \ \ & \text{ if }x\in \gamma_3(G).          
    \end{array}\]

\end{Prop}

 \begin{Prop}[Lemma 2.3, \cite{Lewis1}]\label{prop:cammod}
    Let $G$ be a Camina group.
    \begin{enumerate}[leftmargin=*]
        \item If $N$ is a normal subgroup of $G$, then either $N\leq G'$ or $G'\leq N$.
        \item If $N<G'$ is a normal subgroup of $G$, then $G/N$ is also a Camina group.
    \end{enumerate}
\end{Prop}

 \begin{Lem}[Lemma 4.3, \cite{Lewis1}]\label{lem:dercam}
            Let $G$ be a Camina group. Then $Z(G)\leq G'$.
        \end{Lem}

If $H\normal G$, then $(G,H)$ is a \emph{Camina pair} if for $g\in G\setminus H$, $g$ is conjugate to every element of $gH$. We have:

\begin{Lem}[Lemma 4.1, \cite{Lewis1}]\label{lem:pair}
    Let $1<K\triangleleft G$. Then the following are equivalent:
    \begin{enumerate}
        \item $(G,K)$ is a Camina pair.
        \item If $x\in G\setminus K$, then $|C_G(x)|=|C_{G/K}(xK)|$.
        \item If $xK$ and $yK$ are conjugate and nontrivial in $G/K$, then $x$ is conjugate to $y$ in $G$.
        \item If $C_1=\{1\},C_2,\cdots , C_m$ are the conjugacy classes of $G$ contained in $K$ and $C_{m+1},\cdots, C_n$ are the conjugacy classes of $G$ outside $K$, then $C_iC_j=C_j$ for $1\leq i\leq m$ and $m+1\leq j\leq n$.
    \end{enumerate}
\end{Lem}

 \begin{Lem}[\cite{Lewis1}]\label{lem:camsub}
        Let $(G,K)$ be a Camina pair. Then:
        \begin{enumerate}
            \item If $N<K$ is normal in $G$, then $(G/N,K/N)$ is a Camina pair.
            \item $Z(G)\leq K\leq G'$.
        \end{enumerate}
    \end{Lem}

\begin{Prop}[\cite{Lewis1}]\label{prop:campairgroup}
    A group $G$ is a Camina group if and only if $(G,G')$ is a Camina pair.
\end{Prop}

A group $G$ is called a \emph{Con-cos} group \cite{Con-cos1,Con-cos2} if there is $K\normal G$ with $xK=x^G$ for $x\in G\setminus K$.

\begin{Prop}[Proposition 2.8, \cite{Few_Conj}]\label{prop:concos}
    Let $G$ be a group. The following conditions are equivalent:
    \begin{enumerate}
        \item $G$ is Con-cos;
        \item $G$ is either a Camina group, or an abelian group.
        \item $xG'=x^G$ for all $x\in G\setminus G'$;
        \item $G'=\{[x,y]\colon y\in G\}$ for any $x\in G\setminus G'$.
    \end{enumerate}
\end{Prop}

\begin{Cor}\label{cor:consolv}
    If $G$ is a Con-cos group, then $G$ is solvable.
\end{Cor}

We now mention the classification from \cite{Bastian2}.

\begin{Thm}\label{thm:acclassify}[Theorem 4.1, \cite{Bastian2}]
        $T(G)$ is AC if and only if $G$ is isomorphic to one of:
        \begin{enumerate}
        \item A finite abelian group.
        \item A Frobenius group $(\Z_p)^r\rtimes \Z_{p^{r}-1}$, for some prime $p$ and $r>0$.
        \item A non-abelian Camina $p-$group, for some prime $p$.
        \item The Frobenius group $\Z_3^2\rtimes Q_8$.
    \end{enumerate}
    \end{Thm}

\section{AC Terwilliger algebras for SGPs}

\begin{Lem}\label{lem:centralizerh}
    Let $(G,H)$ be a SGP. Suppose for all $x,y\in G$ if $x^H\neq (y^{-1})^H$ then $x^H\cdot y^H=(xy)^H$. Then $x^HC_G(H)=x^H$ for all $x\in G\setminus C_G(H)$. Also, $C_G(H)\leq [K,H]$ whenever $K\subseteq G$ is not a subset of $C_G(H)$.
\end{Lem}

\begin{proof}
    Let $x\in G\setminus C_G(H)$ and $z\in C_G(H)$. Then $x^H\neq (z^{-1})^H$, so $z^Hx^H=(zx)^H$. Suppose that 
    \[x^H\neq x^Hz=x^Hz^H=z^Hx^H=(zx)^H=((zx)^{-1})^{-1})^H.\]
    Then the hypothesis gives $x^H((zx)^{-1})^H=(xx^{-1}z^{-1})^H=(z^{-1})^H=\{z^{-1}\}$. 
    But $|x^H|>1$ and so this is a contradiction. Thus, $x^H=x^Hz$ for $z\in C_H(G)$, so $x^H=x^HC_G(H)$. 


    Now let $K\subseteq G$ such that $K\not\subseteq C_G(H)$. Fix $x\in K\setminus C_G(H)$. Then 
    \[C_G(H)=[x,e]C_G(H)\subseteq [x,H]C_G(H)=x^{-1}x^HC_G(H)=x^{-1}x^H=[x,H]\subseteq [K,H].\qedhere\]
\end{proof}

We now give necessary conditions for when $(G,H)$ is a SGP and $T(G,\C[G]^H)$ is AC.

\begin{Thm}\label{thm:sgpequiv1}
    Suppose that $G$ is nonabelian, and $(G,H)$ is a SGP. If $T(G,\C[G]^H)$ is AC then: 
    \begin{enumerate}[label={\normalfont(\arabic*)}]
        \item $T(H)$ is AC;
         \item $H\normal G$;
         \item $C_G(H)\leq H$;
        \item  $G/H$ is abelian.
    \end{enumerate}
\end{Thm}

\begin{proof}
    $(1)$: Assume $T(G,\C[G]^H)$ is AC. If $T(H)$ is not AC, then by Theorem \ref{thm:sequiv}, there are $x,y\in H$ such that $x^H\neq (y^{-1})^H$ and $x^H\cdot y^H\neq (xy)^H$. Then by Theorem \ref{thm:sequiv}, $T(G,\C[G]^H)$ is not AC. Thus, $T(H)$ is AC.

    $(2)$: We show that $gH=Hg$ for all $g\in G\setminus H$. We have two cases.\\
    {\bf Case 1: } $gH\subseteq g^H$. Then $|gH|\leq |g^H|\leq |H|=|gH|$, so $|gH|=|g^H|$ and $gH=g^H$. Now let $h\in H$. As $gh\in gH=g^H$, we have $(hg)^H=(gh)^H=g^H=gH$. So $hg\in gH$. Thus $Hg\subseteq gH$ and $Hg=gH$.\\
    {\bf Case 2: }$gH\not\subseteq g^H$. Then there exists $h\in H$ such that $gh\not\in g^H$, and so $(gh)^H\neq g^H$. By Theorem \ref{thm:sequiv}, 
    \[(ghg^{-1})^H=(gh)^H\cdot (g^{-1})^H=(g^{-1})^H\cdot (gh)^H=h^H.\] 
    
    Since $(gh)^H\neq g^H$ and $h^H\neq (g^{-1})^H$ by Theorem \ref{thm:sequiv}, 
    \[h^H=(ghg^{-1})^H=(gh)^H\cdot (g^{-1})^H=g^H\cdot h^H\cdot (g^{-1})^H=g^H\cdot (g^{-1})^H\cdot h^H.\]
    Then for $y\in g^H\cdot (g^{-1})^H$ we have $yh\in g^H\cdot (g^{-1})^H\cdot h^H=h^H\subseteq H$. Thus, $y\in H$ and $g^H\cdot (g^{-1})^H\subseteq H$. 
    
    For all $s\in H$, $gsg^{-1}s^{-1}\in g^H\cdot (g^{-1})^H\subseteq H$. Then $gsg^{-1}\in H$, and so $gHg^{-1}\subseteq H$. This implies $Hg=gH$, and so $H\normal G$.
    
    $(3)$: First we show $G\not\subseteq C_G(H)$. Suppose $G\subseteq C_G(H)$. Then for $g\in G$, $g^H=\{g\}$. Thus for $x,y\in G$ we have $xy=x^H\cdot y^H=y^H\cdot x^H=yx$, and so $G$ is abelian, a contradiction. Hence, $G\not\subseteq C_G(H)$. By Lemma \ref{lem:centralizerh}, $C_G(H)\leq [G,H]$. As $H\normal G$, we have $[G,H]\leq H$, so $C_G(H)\leq H$. 

   $(4)$: Fix $g\in G$. Let $y=hgh^{-1}\in g^H$ for $h\in H$. Then $y=hgh^{-1}\in HgH=gHH=gH$, so $g^H\subseteq gH$ for $g\in G$. Now let $xH,yH\in G/H$. We show $xH$ and $yH$ commute. There are two cases:\\
   {\bf Case 1: } $x^H=(y^{-1})^H$. Then $y^{-1}\in x^H\subseteq xH$ and $xH=y^{-1}H$. Thus $(xH)(yH)=(y^{-1}H)(yH)=H=(yH)(y^{-1}H)=(yH)(xH)$.\\
   {\bf Case 2: } $x^H\neq (y^{-1})^H$. By Theorem \ref{thm:sequiv}, $(xy)^H=x^Hy^H=y^Hx^H=(yx)^H$. 
    Hence, $yx\in (xy)^H\subseteq xyH$. However, $yx\in yxH$ and so $yxH=xyH$. Thus, $(xH)(yH)=xyH=yxH=(yH)(xH)$, and so $G/H$ is abelian.  
\end{proof}

\begin{Thm}\label{thm:sgpequivab}
     Suppose that $G$ is nonabelian and $(G,H)$ is a SGP. If $T(G,\C[G]^H)$ is AC and $H$ is abelian then  for $x\in G\setminus H$, we have $x^H=xH$.
\end{Thm}

\begin{proof}
    By Theorem \ref{thm:sgpequiv1}, $C_G(H)\leq H$. Since $H$ is abelian $H\subseteq C_G(H)$ and we have $C_G(H)=H$.

    Let $g\in G\setminus H$. By Lemma \ref{lem:centralizerh}, $H=C_G(H)\leq [g,H]$. Fix $h\in H$. Then, $h=g^{-1}y^{-1}gy$ for some $y\in H$, and so $gh=y^{-1}gy$. Thus, $gH\subseteq g^H$. Then as $|g^H|\leq |H|=|gH|$ we have $g^H=gH$. 
\end{proof}

\begin{Prop}\label{prop:sgpmod}
    Suppose $G$ is nonabelian and $(G,H)$ is a SGP. Assume $T(G,\C[G]^H)$ is AC. Then $(G/H',H/H')$ is a SGP and $T\left(G/H',\C[G/H']^{H/H'}\right)$ is AC.
\end{Prop}

\begin{proof}
    By Theorem \ref{thm:sgpequiv1}, $H\normal G$. As $H'$ is a characteristic subgroup of $H$, $H'\normal G$. Then by Lemma \ref{lem:sgpmod}, $(G/H',H/H')$ is a SGP, so $\C[G/H']^{H/H'}$ is a commutative Schur ring.

    Now let $\overline{G}=G/H'$ and $\overline{H}=H/H'$. For $g\in G$ we let $\overline{g}=gH'\in \overline{G}$. Let $\overline{x},\overline{y}\in \overline{G}$ such that $\overline{x}^{\overline{H}}\neq (\overline{y^{-1}})^{\overline{H}}$. If $x^H=(y^{-1})^H$, there is a $h\in H$ such that $x=hy^{-1}h^{-1}$. Then
    \[\overline{x}=xH'=hy^{-1}h^{-1}H'=hH'h^{-1}H'y^{-1}h^{-1}H'=\overline{h}\cdot\overline{y^{-1}}\overline{h^{-1}}\in (\overline{y^{-1}})^{\overline{H}}.\]
    This implies $\overline{x}^{\overline{H}}= (\overline{y^{-1}})^{\overline{H}}$, which is false, so $x^H\neq (y^{-1})^H$. Then by Theorem \ref{thm:sequiv}, $x^H\cdot y^H=(xy)^H$. 

    We claim $\overline{x}^{\overline{H}}\cdot\overline{y}^{\overline{H}}=(\overline{xy})^{\overline{H}}$. Notice that $\overline{xy}\in \overline{x}^{\overline{H}}\overline{y}^{\overline{H}}$, so $(\overline{xy})^{\overline{H}}\subseteq \overline{x}^{\overline{H}}\cdot\overline{y}^{\overline{H}}$. Now let $\overline{z_1}\in \overline{x}^{\overline{H}}$, $\overline{z_2}\in \overline{y}^{\overline{H}}$. Then $\overline{z_1}=\overline{h_1}\overline{x}\overline{h_1^{-1}}$ and $\overline{z_2}=\overline{h_2}\overline{y}\overline{h_2^{-1}}$ for $\overline{h_1},\overline{h_2}\in \overline{H}$. Notice that $h_1xh_1^{-1}\in x^H$ and $h_2yh_2^{-1}\in y^H$. Then $(h_1xh_1^{-1})(h_2yh_2^{-1})\in x^Hy^H=(xy)^H$. Therefore there is a $h_3\in H$ such that $(h_1xh_1^{-1})(h_2yh_2^{-1})=h_3xyh_3^{-1}$. Then
    \begin{align*}
\overline{z_1}\overline{z_2}&=\left(\overline{h_1}\overline{x}\overline{h_1^{-1}}\right)\left(\overline{h_2}\overline{y}\overline{h_2^{-1}}\right)=\left(h_1xh_1^{-1}H'\right)\left(h_2yh_2^{-1}H'\right)=\left(h_1xh_1^{-1}h_2yh_2^{-1}\right) H'\\
    &=\left(h_3xyh_3^{-1}\right)H'=h_3H'(xy)H'h_3^{-1}H'=\overline{h_3}\overline{xy}\overline{h_3^{-1}}\in (\overline{xy})^H.
        \end{align*}
    So $\overline{x}^{\overline{H}}\cdot \overline{y}^{\overline{H}}\subseteq (\overline{xy})^{\overline{H}}$, and thus $\overline{x}^{\overline{H}}\cdot \overline{y}^{\overline{H}}= (\overline{xy})^{\overline{H}}$. Therefore, for all $\overline{x},\overline{y}\in \overline{G}$ with $\overline{x}^{\overline{H}}\neq (\overline{y^{-1}})^{\overline{H}}$, we have $\overline{x}^{\overline{H}}\cdot \overline{y}^{\overline{H}}= (\overline{xy})^{\overline{H}}$. Then by Theorem \ref{thm:sequiv}, $T\left(G/H',\C[G/H']^{H/H'}\right)$ is AC.  
\end{proof}

\begin{Thm}\label{thm:sgpequiv}
    Let $G$ be nonabelian with $H\leq G$ and $(G,H)$ be a SGP. If $T(G,\C[G]^H)$ is AC then for $x\in G\setminus H$, we have $x^H=xH$.
\end{Thm}

\begin{proof}
    Fix $g\in G\setminus H$. By Theorem \ref{thm:sgpequiv1}, $H\normal G$ which implies $g^H\subseteq HgH= gH$. Now suppose $h\in H$ such that $gh\not\in g^H$. Then $(gh)^H\neq g^H$. By Theorem \ref{thm:sequiv}, 
    \[(ghg^{-1})^H=(gh)^H(g^{-1})^H=(g^{-1})^H(gh)^H=h^H.\] 
    However we note, since $g\in G\setminus H$, that $g^H\neq (h^{-1})^H$, so by Theorem \ref{thm:sequiv}, we also have
    \[h^H=(ghg^{-1})^H=(gh)^H(g^{-1})^H=g^Hh^H(g^{-1})^H=g^H(g^{-1})^Hh^H.\]

    Let $x\in g^H\cdot (g^{-1})^H$. As $\C[G]^H$ is a Schur ring we can write $\overline{g^H}\cdot \overline{(g^{-1})^H}=\sum \lambda_k \overline{y_k^H}$ for some $y_k\in G$ and nonzero integers $\lambda_k$. Then
    \[\overline{g^H}\cdot \overline{(g^{-1})^H}\cdot \overline{h^H}=\sum \lambda_k \overline{y_k^H}\cdot \overline{h^H}.\]
    However, as $h^H=g^H(g^{-1})^Hh^H$, we know that $\overline{g^H}\cdot \overline{(g^{-1})^H}\cdot \overline{h^H}=\mu \overline{h^H}$ for some nonzero integer $\mu$. We thus have, $\mu \overline{h^H}=\sum \lambda_k \overline{y_k^H}\cdot \overline{h^H}$. This implies $h^H=(y_k^H)h^H$ for all $y_k^H$ with a nonzero coefficient $\lambda_k$ in $\sum \lambda_k \overline{y_k^H}$. We know $x\in g^H(g^{-1})^H$, so $x^H=y_i^H$ for some $i$. Therefore, $h^H=x^H\cdot h^H$. Then $xh=tht^{-1}$ for some $t\in H$. This implies $x=tht^{-1}h^{-1}\in H$. So $x\in H$. We claim that $x\in H'$. Suppose not. By Theorem \ref{thm:sgpequiv1}, $T(H)$ is AC, so by Theorem \ref{thm:acclassify}, $H$ is a Con-cos group. Therefore, for all $a\in H\setminus H'$, $a^H=aH'$. In particular, $x^H=xH'$. We consider two cases.\\
    {\bf Case 1: }$h\in H'$. Then $h^H\subseteq H'$, and $x^H\cdot h^H=xH'\cdot h^H=xH'$. However, $xH'\neq h^H$ contradicting $x^H\cdot h^H=h^H$.\\
    {\bf Case 2: }$h\not\in H'$. Then $h^H=hH'$, and $x^H\cdot h^H=xH'\cdot hH'=xhH'$. Notice $xhH'\neq hH'$ since $x\not\in H'$. Therefore, $x^H\cdot h^H\neq h^H$ in this case, contradicting $x^H\cdot h^H=h^H$.\\
    In both cases we get a contradiction, so $x\in H'$. As we chose $x\in g^H\cdot(g^{-1})^H$ arbitrarily, we have $g^H\cdot (g^{-1})^H\subseteq H'$. For all $y\in g^H$, we have $yg^{-1}\in g^H\cdot (g^{-1})^H\subseteq H'$. Then $y\in H'g$. Since $H\normal G$ we have $H'\normal G$. So $H'g=gH'$. Thus $y\in gH'$. As we chose $y\in g^H$ arbitrarily $g^H\subseteq gH'$, so $g^HH'\subseteq gH'$. 

    By Proposition \ref{prop:sgpmod}, $(G/H',H/H')$ is a SGP with $T(G/H',\C[G/H']^{H/H'})$ AC. Set $G/H'=\overline{G}$, $H/H'=\overline{H}$, and $\overline{x}=xH'$ for all $x\in G/H'$. Notice that $\overline{H}$ is abelian so by Theorem \ref{thm:sgpequivab}, for all $\overline{a}\in \overline{G}\setminus \overline{H}$ we have $\overline{a}\overline{H}=\overline{a}^{\overline{H}}$. Then as $g\in G\setminus H$, $\overline{g}\in \overline{G}\setminus \overline{H}$ and $\overline{g}\overline{H}=\overline{g}^{\overline{H}}$. We observe that $\overline{gh}\in \overline{g}\overline{H}=\overline{g}^{\overline{H}}$. Then there exists some $\overline{s}\in \overline{H}$ such that $\overline{gh}=\overline{s}\overline{g}\overline{s^{-1}}$, and $ghH'=sgs^{-1}H'$. So 
    \[gh\in (sgs^{-1})H'\subseteq (sgs^{-1})^HH'=g^HH',\]
    since $s\in H$. Therefore, $gh\in g^HH'\subseteq gH'$ by the above argument. So $h\in H'$. 
    
    Therefore, for all $h\in H$, if $gh\not\in g^H$, then $h\in H'$. Suppose $x\in H\setminus H'$. If $gx\not\in g^H$ then we have $x\in H'$, which is false. Hence, for $x\in H\setminus H'$, we have $gx\in g^H$.

    Recall that $H$ is a Con-cos group. Then $H$ is solvable by Corollary \ref{cor:consolv}, so there exists some $z\in H\setminus H'$ and $gz\in g^H$. Then as $g\in G\setminus H$, and $z\in H$ we have by Theorem \ref{thm:sequiv} that $g^H=(gz)^H=g^Hz^H$. Since $z\in H\setminus H'$ and $H$ is a Con-cos group $z^H=zH'$. Therefore, $g^H=g^HzH'=g^HH'z$. Then for all $h\in H'$, $ghz\in g^H$. So there exists some $y_h\in H$ such that $ghz=y_hgy_h^{-1}$. Then $gh=y_hgy_h^{-1}z^{-1}\in g^H\cdot (z^{-1})^H$. As $z^{-1}\in H$ and $g\in G\setminus H$, we have by Theorem \ref{thm:sequiv} that $g^H\cdot (z^{-1})^H=(gz^{-1})^H$. So $gh\in (gz^{-1})^H$. However, $z^{-1}\in H\setminus H'$ as $z\in H\setminus H'$. Then $gz^{-1}\in g^H$ by the above paragraph. So $(gz^{-1})^H=g^H$. Thus $gh\in g^H$, for all $h\in H'$, we have $gh\in g^H$. 

    Now let $ga\in gH$ for any $a\in H$. If $a\in H'$, then we just found $ga\in g^H$. If $a\in H\setminus H'$ we found above that $ga\in g^H$. Therefore, $ga\in g^H$ for all $a\in H$, so $gH\subseteq g^H$. Hence, $g^H=gH$. 
\end{proof}

Thus if $(G,H)$ is a SGP and $T(G,\C[G]^H)$ is AC, then $H\normal G$, $T(H)$ is AC, $G/H$ is abelian, and $x^H=xH$ for all $x\in G\setminus H$. We shall now prove that if these four conditions are met, then $(G,H)$ is a SGP and $T(G,\C[G]^H)$ is AC.

\begin{Prop}\label{thm:sgp}
    Let $G$ be a group. Suppose $H\normal G$ and $G/H$ is abelian. If for all $x\in G\setminus H$ we have $x^H=xH$, then $(G,H)$ is a SGP. 
\end{Prop}

\begin{proof}
The classes $\{g^H\}$, determine an $S-$ring. We only need to check commutativity. 

Let $x,y\in H$ and let $x^H=P_x$, $y^H=P_y$. Then as $\mathcal{G}(H)$ is commutative we have $\overline{P_x}\cdot \overline{P_y}=\overline{P_y}\cdot \overline{P_x}$.

Next suppose $x,y\in G\setminus H$. Then $x^H=xH$ and $y^H=yH$. As $G/H$ is abelian, we have 
\[\overline{x^H}\cdot \overline{y^H}=\overline{xH}\cdot \overline{yH}=|H|\overline{xyH}=|H|\overline{yxH}=\overline{yH}\cdot \overline{xH}=\overline{y^H}\cdot \overline{x^H}.\]

Suppose now that $x\in H$ and $y\in G\setminus H$. Let $x^H=P_x\subseteq H$ and $y^H=yH=P_y$. Then
\[\overline{P_x}\cdot y\overline{H}=\overline{P_x}\cdot\overline{H}y=|P_x|\overline{H}y=|P_x|y\overline{H}=|P_x|\overline{P_y}.\]
By similar reasoning, $\overline{yH}\cdot \overline{P_x}=|P_x|\overline{P_y}$.
Hence, $(G,H)$ is a SGP.
\end{proof}

\begin{Cor}\label{lem:sgpac}
     Let $G$ be a group. Suppose $H\normal G$, $T(H)$ is AC, and $G/H$ is abelian. If for all $x\in G\setminus H$ we have $x^H=xH$, then $(G,H)$ is a SGP and $T(G,\C[G]^H)$ is AC.
\end{Cor}

\begin{proof}
    By Proposition \ref{thm:sgp}, $(G,H)$ is a SGP, so $\C[G]^H$ is commutative. Let the principal classes be $P_0,P_1,\cdots, P_d$. 
    
    Now let $x,y\in G$ with $x\in P_i$, $y\in P_j$ and $xy\in P_k$ such that $P_i\neq P_j^*$. If $x,y\in H$ then $P_i=x^H$ and $P_j=y^H$. Since $T(H)$ is AC and $x^H=P_i\neq P_j^*=(y^{-1})^H$ we have by Theorem \ref{thm:sequiv} that $P_iP_j=x^Hy^H=(xy)^H=P_k$. 

    If $x,y\in G\setminus H$, then $P_i=xH$, $P_j=yH$, and $P_iP_j=xyH$. As $P_i\neq P_j^*$, we have $yH\neq x^{-1}H$, so $xy\not\in H$. Thus, $xyH=(xy)^H=P_k$ and $P_iP_j=P_k$.

    If $x\in H$ and $y\in G\setminus H$, then $P_i=x^H$ and $P_j=yH$. So we have $P_iP_j=x^H\cdot yH=yH$ since $x^H\subseteq H$. Notice that $xy\in Hy=yH$. Also $xy\not\in H$ as $y\not\in H$. So $P_k=(xy)^H=xyH=yH=P_iP_j$.

    From these three cases we see that $P_iP_j=P_k$. So by Theorem \ref{thm:sequiv}, $T(G,\C[G]^H)$ is AC. 
\end{proof}

At this point we consider which pairs $(G,H)$ satisfy our conditions that $T(H)$ is AC, $G/H$ is abelian, and $x^H=xH$ for all $x\in G\setminus H$. 

\begin{Prop}\label{prop:frobca}
    Let $G$ be nonabelian. Suppose $H\triangleleft G$ is a proper subgroup, that $G/H$ is abelian and for $x\in G\setminus H$ we have $x^H=xH$; then $G$ is a Frobenius group with cyclic complement, $H=G'$, and $H$ is the Frobenius kernel of $G$.
\end{Prop}

\begin{proof}
As $G/H$ is abelian, if $xH$ and $yH$ are conjugate in $G/H$, then $xH=yH$. If additionally, $x,y\not\in H$ then $xH=x^H$ and $yH=y^H$, so $x^H=y^H$. Then by Lemma \ref{lem:pair}, $(G,H)$ is a Camina pair and by Lemma \ref{lem:camsub}, $Z(G)\leq H\leq G'$. However, as $G/H$ is abelian $G'\leq H$. Therefore, $H=G'$ and $G$ is a Camina group.

First assume that $G$ is a Camina $p-$group. Then by Theorem \ref{thm:camclassify}, the nilpotency class of $G$ is either $2$ or $3$. This gives us two cases.\\
{\bf Case 1: }Suppose the nilpotency class of $G$ is $2$. Then $Z(G)=G'=H$. Let $g\in G\setminus H$, so $g^H=\{g\}$. However, $\{g\}\neq gH$. This contradicts our assumption that for $x\in G\setminus H$ we have $x^H=xH$.\\
{\bf Case 2: }Suppose the nilpotency class of $G$ is $3$. Then $Z(G)\lneq G'$. Fix $g\in G\setminus H$, $h\in H=G'$. If $h\in Z(G)$, then $hgh^{-1}=g\in gZ(G)$. If $h\in H\setminus Z(G)$ then by Proposition \ref{lem:pcamclass}, $h^G=hZ(G)$. Therefore, $ghg^{-1}=hz$ for some $z\in Z(G)$. So $h^{-1}gh=zg=gz$ and $h^{-1}gh\in gZ(G)$. So for all $h\in H$, $h^{-1}gh\in gZ(G)$. Therefore, $g^H\subseteq gZ(G)$. However, $gZ(G)\subsetneq gH$. Thus, $g^H\neq gH$, a contradiction

Hence, $G$ cannot be a Camina $p-$group. By Theorem \ref{thm:camclassify}, $G$ is either $(i)$ a Frobenius group with cyclic complement or $(ii)$ a Frobenius group with complement $Q_8$. Suppose we have $(ii)$ with complement $K\cong Q_8$. As $Q_8$ is solvable, so is $G$, and $G\neq G'$. By Proposition \ref{prop:cammod}, either $N\leq G'$ or $G'\leq N$. If $G'\leq N$, then $G/N$ is abelian. However, $G/N\cong Q_8$ is nonabelian, so $N<G'$. 

Fix $x\in G\setminus G'$. By assumption $x^{G'}=xG'$. As $x\not\in N$, by Proposition \ref{prop:frobeq}, there exists some $h\in K$ with $x=ghg^{-1}$. Now as $G=NK$, $g=nk$ for some $n\in N$, $k\in K$. So $x=nkhk^{-1}n^{-1}$. Note that $h_1=khk^{-1}\in K$. Then $x=nh_1n^{-1}$, and $x^{G'}=h_1^{G'}$ as $n\in N\subseteq G'$.

Now let $y\in h_1^{G'}$. Then there is $s\in G'$ such that $y=sh_1s^{-1}$. We know $N< G'\normal G$ and $N\normal G$. Then $G'/N\normal G/N\cong Q_8$ and $|G\colon G'|=|G/N\colon G'/N|$. Since $G'$ is a proper subgroup of $G$, we have $|G\colon G'|>1$. Then $|G/N\colon G'/N|>1$. Hence, $G'/N$ is a proper subgroup of $G/N$. As every proper subgroup of $Q_8$ is abelian, we know $G'/N$ is abelian. Then $(sN)(h_1N)(sN)^{-1}=sh_1s^{-1}N=h_1N$. Thus, $y=sh_1s^{-1}\in h_1N$. So $h_1^{G'}\subseteq h_1N$. Thus, $x^{G'}\subseteq h_1N$ as $x^{G'}=h_1^{G'}$. However, by assumption, $x^{G'}=xG'$, so $xG'\subseteq h_1N$ and $|xG'|\leq |h_1N|$. However, $|xG'|=|G'|$, $|h_1N|=|N|$, and $|N|<|G'|$ since $N$ is a proper subgroup of $G'$. So $|xG'|>|h_1N|$ which contradicts $|xG'|\leq |h_1N|$.

Hence, $G$ is a Frobenius group with cyclic complement $K$ and Frobenius kernel $N$. By Proposition \ref{prop:frobeq}, for all $n\in N$ there is a $g\in G$ such that $ghg^{-1}=nh$ for some $h\in K$. This implies $n=[g^{-1},h^{-1}]\in G'$, so $N\subseteq G'$. However, $G/N\cong K$ and $K$ is abelian. Therefore, $G'\subseteq N$. Hence, $G'=N=H$ is the Frobenius kernel of $G$.
\end{proof}

\begin{Prop}\label{prop:frobcon}
    Let $G$ be a Frobenius group with abelian complement $H$ and Frobenius kernel $N$. Then $N=G'$ and for all $x\in G\setminus N$ we have $x^N=xN$. 
\end{Prop}

\begin{proof}
    As $G/N\cong H$ is abelian, we have $G'\leq N$. Let $n\in N$. By Proposition \ref{prop:frobeq}, there is $g\in G$ such that $ghg^{-1}=nh$, $h\in H$. Thus $n=ghg^{-1}h^{-1}=[g^{-1},h^{-1}]\in G'$, and so $N=G'$. 
    
    Now let $x\in G\setminus N$. Then by Proposition \ref{prop:frobeq}, there is $h\in H$ such that $x$ is conjugate to $h$. Suppose $g\in G$ such that $x=ghg^{-1}$. As $G=NH$, $g=nk$ for some $n\in N$ and $k\in H$. So $x=ghg^{-1}=nkhk^{-1}n^{-1}$. However, as $H$ is abelian, $khk^{-1}=h$. Thus, $x=nhn^{-1}$. So $x\in h^N$. As the classes of conjugation by elements of $N$ partition $G$, this implies $x^N=h^N$. As $h\in H\setminus \{e\}$ we have by Proposition \ref{prop:frobeq}, that $h$ is conjugate to every element of $hN$. Let $y\in hN$. Then there exists $r\in G$ such that $y=rhr^{-1}$. As $G=NH$, $r=ms$ for some $m\in N$ and $s\in H$. So $y=rhr^{-1}=mshs^{-1}m^{-1}$. However, as $H$ is abelian, $shs^{-1}=h$. Thus, $y=mhm^{-1}\in h^N$ and $hN\subseteq h^N$. Given $z\in h^N$, we have $z=tht^{-1}$ for some $t\in N$. Then $z\in NhN=hN$. Thus, $z\in hN$ and $h^N\subseteq hN$. Hence, $h^N=hN$. Recall that $x^N=h^N$, so $x^N=hN$. As $x\in hN$, $hN=xN$. Therefore, $x^N=xN$.   
\end{proof}

\begin{Cor}\label{cor:frobsgp}
    Let $G$ be a Frobenius group with abelian complement and Frobenius kernel $N$. Then $(G,N)$ is a SGP.
\end{Cor}

\begin{proof}
    By Proposition \ref{prop:frobcon}, $N=G'$ and for $x\in G\setminus N$, $x^N=xN$. Then by Proposition \ref{thm:sgp}, $(G,N)$ is a SGP.
\end{proof}

We can now classify all SGPs that result in an AC Terwilliger algebra:

\begin{Thm}\label{thm:sgpclassify}
    Let $H\lneq G$. Then $(G,H)$ is a SGP and $T(G,\C[G]^H)$ is AC if and only if 
    \begin{enumerate}
        \item $G$ is abelian with $H$ being any subgroup of $G$, or
        \item $G$ is a Frobenius group with Frobenius kernel $H=G'$ and cyclic complement such that $H$ is either abelian or a Camina $p-$group.
    \end{enumerate}
     
\end{Thm}

\begin{proof}
    First suppose $(G,H)$ is a SGP, $T(G,\C[G]^H)$ is AC. If $G$ is abelian, then we are done. Suppose $G$ is nonabelian. Then by Theorems \ref{thm:sgpequiv1} and \ref{thm:sgpequiv}, $T(H)$ is AC, for all $x\in G\setminus H$ we have $x^H=xH$, and $G/H$ is abelian. As $G$ is nonabelian, Proposition \ref{prop:frobca} gives that $G$ is a Frobenius group with cyclic complement and $H=G'$ is the Frobenius kernel. Since $T(H)$ is AC, by Theorem \ref{thm:acclassify} $H$ is either abelian, a Camina $p-$group, a Frobenius group of the form $\Z_p^r\rtimes \Z_{p^r-1}$, or $\Z_3^2\rtimes Q_8$. However, by Theorem \ref{thm:frobnil}, $H$ must be a nilpotent group since it is the Frobenius kernel. Since Frobenius groups are not nilpotent either $H$ is abelian or a Camina $p-$group.
    
    Therefore, $G$ is a Frobenius group with Frobenius kernel $H=G'$ and cyclic complement such that $H$ is either abelian or a Camina $p-$group. This proves the forward direction.

    Now we prove the reverse direction. If $G$ is abelian and $H\leq G$, then $\C[G]^H=\C[G]$, which is commutative. Thus, $(G,H)$ is a SGP and by Theorem \ref{thm:acclassify}, $T(G,\C[G]^H)$ is AC. 

    If $G$ is a Frobenius group with Frobenius kernel $H$ and cyclic complement such that $H$ is abelian or a Camina $p-$group, then by Theorem \ref{thm:acclassify}, $T(H)$ is AC. By Proposition \ref{prop:frobcon}, $H=G'$ is the Frobenius kernel of $G$ and for all $x\in G\setminus H$ we have $x^H=xH$. Then by Corollary \ref{lem:sgpac}, $(G,H)$ is a SGP and $T(G,\C[G]^H)$ is AC.
\end{proof}

We now give two examples of infinite families of Frobenius groups satisfying Theorem \ref{thm:sgpclassify}.

\begin{Examp}
    \normalfont{Let $H$ be the extraspecial group of exponent $p>2$ and order $p^3$:
    \[H=\langle a,b,c\colon a^p=b^p=c^p=1,c=a^{-1}b^{-1}ab,ac=ca,bc=cb\rangle.\]
    Let $k\in \Z$, $k>2$, where $k\mid (p-1)$ is prime and $k^2\not\equiv 1 \mod p$. Then $\varphi\colon H\rightarrow H$ given by $\varphi(a)=a^k,\ \varphi(b)=b^k,\ \varphi(c)=c^{k^2}$,
    determines an automorphism that only fixes $1$. Therefore, if $d$ is the order of $k$ in $\Z_p^\times$, then $G=H\rtimes_\varphi \Z_d$ is a Frobenius group with extraspecial kernel and cyclic complement, so that $T(G,\C[G]^H)$ is AC.} \hfill$\square$
\end{Examp}

\begin{Examp}
\normalfont{
    Let $p>2$ be prime. Let $H=\langle h_1,\cdots, h_6\rangle$ where
    \[h_2^{h_1}=h_2h_4,\ h_3^{h_1}=h_3h_5,\ h_3^{h_2}=h_3h_6\]
    with $h_i^p=1$, $1\leq i\leq 6$, and $Z(H)=\langle h_4,h_5,h_6\rangle$. Then $H$ is a Camina group of order $p^6$ with nilpotency class $2$. Let $X=(x_{ij})\in \SL(3,p)$ have prime order not equal to $p$ or $2$. We want $\varphi\in \Aut(H)$ such that
    \[g_4: =\varphi(h_4)=h_4^{x_{11}}h_5^{x_{21}}h_6^{x_{31}},\ g_5: =\varphi(h_5)=h_4^{x_{12}}h_5^{x_{22}}h_6^{x_{32}},\ g_6:=\varphi(h_6)=h_4^{x_{13}}h_5^{x_{23}}h_6^{x_{33}},\]
    and $a_i,b_i,c_i\in \Z_p$ with
    \begin{equation}\label{eq:x'}
      g_1=h_1^{a_1}h_2^{a_2}h_3^{a_3},\ g_2=h_1^{b_1}h_2^{b_2}h_3^{b_3},\ g_3=h_1^{c_1}h_2^{c_2}h_3^{c_3},  
    \end{equation}
    of order $p$ such that $g_2^{g_1}=g_2g_4$, $g_3^{g_1}=g_3g_5$, $g_3^{g_2}=g_3g_6$. Note that if $h_j^{h_i}=h_jh_k$, then $h_i^{h_j}=h_ih_k^{-1}$. Now,
\begin{align*}
    g_2^{g_1}&=(h^{b_1}h_2^{b_2}h_3^{b_3})^{h_1^{a_1}h_2^{a_2}h_3^{a_3}}=(h_1^{b_1})^{h_2^{a_2}h_3^{a_3}}(h_2^{b_2})^{h_1^{a_1}h_2^{a_2}h_3^{a_3}}(h_3^{b_3})^{h_1^{a_1}h_2^{a_2}h_3^{a_3}}\\
    &=h_1^{b_1}h_4^{a_1b_2-a_2b_1}h_5^{a_1b_3-a_3b_1}h_6^{a_2b_3-a_3b_2}\cdots=g_2g_4\cdots =g_2h_4^{x_{11}}h_5^{x_{21}}h_6^{x_{31}}\cdots.
\end{align*}
    Similar arguments yield
    \[g_3^{g_1}=h_1^{c_1}h_4^{a_1c_2-c_2b_1}h_5^{a_1c_3-a_3c_1}h_6^{a_2c_3-a_3c_2}\cdots=g_3g_5\cdots =g_3h_4^{x_{12}}h_5^{x_{22}}h_6^{x_{32}}\cdots,\]
     \[g_3^{g_2}=h_1^{c_1}h_4^{b_1c_2-b_2c_1}h_5^{b_1c_3-b_3c_1}h_6^{b_2c_3-b_3c_2}\cdots=g_3g_6\cdots =g_3h_4^{x_{13}}h_5^{x_{23}}h_6^{x_{33}}\cdots.\]
     So we need to solve (over $\Z_p$)
     \begin{equation}\label{eq:x}
        \begin{bmatrix}
            a_1b_2-a_2b_1 & a_1b_3-a_3b_1 & a_2b_3-a_3b_2 \\
            a_1c_2-a_2c_1 & a_1c_3-a_3c_1 & a_2c_3-a_3c_2 \\
            b_1c_2-b_2c_1 & b_1c_3-c_1b_3 & b_2c_3-b_3c_2
        \end{bmatrix}=\begin{bmatrix}
            x_{11} & x_{21} & x_{31}
\\
x_{12} & x_{22} & x_{32}\\
x_{13} & x_{23} & x_{33}
\end{bmatrix}.
     \end{equation}
     Let $A=(a_1,a_2,a_3)^T, B=(b_1,b_2,b_3)^T$, $C=(c_1,c_2,c_3)^T$. Let $\times$ denote the standard cross product on $\mathbb{F}_p^3$. We have
     \[\begin{array}{c}
         A\times B=(a_2b_3-a_3b_2,a_3b_1-a_1b_3,a_1b_2-a_2b_1)^T=(x_{31},-x_{21},x_{11})^T,\ \ \\
        B\times C=(b_2c_3-b_3c_2,b_3c_1-b_1c_3,b_2c_2-b_2c_1)^T=(x_{33},-x_{23},x_{13})^T,\ \ \ \ \\
        C\times A=(c_2a_3-c_3a_2,c_3a_1-c_1a_3,c_1a_2-c_2a_1)^T=(-x_{32},x_{22},-x_{12})^T.
     \end{array}\]
     
     Standard identities are 
     \[A\cdot (B\times C)=B\cdot (C\times A)=C\cdot (A\times B),\]
         \[A\cdot (A\times B)=B\cdot (A\times B)=A\cdot (A\times C)=C\cdot (A\times C)=B\cdot (B\times C)=C\cdot (B\times C)=0.\]
     These gives
     \begin{equation}\label{eq:x3}
     \begin{split}
     & x_{33}a_1-x_{23}a_2+x_{13}a_3=-x_{32}b_1+x_{22}b_2-x_{12}b_3=c_{31}c_1-x_{21}c_2+x_{11}c_3,\\
         & x_{31}a_1-x_{21}a_2+x_{11}a_3=0,\ x_{31}b_1-x_{21}b_2+x_{11}b_3=0,\ 
         x_{32}a_1-x_{22}a_2+x_{12}a_3=0,\\
         & x_{32}c_1-x_{22}c_2+x_{12}c_3=0,\ x_{33}b_1-x_{23}b_2+x_{13}b_3=0,\ x_{33}c_1-x_{23}c_2+x_{13}c_3=0.\ \
              \end{split}
     \end{equation}
     Solving the system of eight linear equations in (\ref{eq:x3}) gives 
     \begin{equation}\label{eq:x4}
         \begin{split}
             a_1=(x_{11}x_{22}-x_{12}x_{21})/(x_{22}x_{33}-x_{23}x_{32})c_3,\\
             a_2=(x_{11}x_{32}-x_{12}x_{31})/(x_{22}x_{33}-x_{23}x_{32})c_3,\\
             a_3=(x_{21}x_{32}-x_{22}x_{31})/(x_{22}x_{33}-x_{23}x_{32})c_3,\\
             b_1=(x_{11}x_{23}-x_{13}x_{21})/(x_{22}x_{33}-x_{23}x_{32})c_3,\\
             b_2=(x_{11}x_{33}-x_{13}x_{31})/(x_{22}x_{33}-x_{23}x_{32})c_3,\\
             b_3=(x_{21}x_{33}-x_{23}x_{31})/(x_{22}x_{33}-x_{23}x_{32})c_3,\\
             c_1=(x_{12}x_{23}-x_{13}x_{22})/(x_{22}x_{33}-x_{23}x_{32})c_3,\\
             c_2=(x_{12}x_{33}-x_{13}x_{32})/(x_{22}x_{33}-x_{23}x_{32})c_3.
         \end{split}
     \end{equation}
Here we are assuming that $x_{22}x_{33}-x_{23}x_{32}\neq 0$. Substituting these values into $a_1b_2-a_2b_1=x_{11}$ gives
\[c_3^2+\frac{(-x_{22}^2x_{33}^2+2x_{22}x_{23}x_{32}x_{33}-x_{23}^2x_{32}^2)}{(x_{11}x_{22}x_{33}-x_{11}x_{23}x_{32}-x_{12}x_{21}x_{33}+x_{12}x_{23}x_{31}+x_{13}x_{21}x_{32}-x_{13}x_{22}x_{31}}=0,\]
which gives $c_3^2=\frac{(x_{22}x_{33}-x_{23}x_{32})^2}{\det X}=(x_{22}x_{33}-x_{23}x_{32})^2$
     as $X\in \SL(3,p)$. Taking $c_3=x_{22}x_{33}-x_{23}x_{32}$ we then get
     \[Y=\begin{bmatrix}
         a_1 & a_2 & a_3 \\
         b_1 & b_2 & b_3 \\
         c_1 & c_2 & c_3
     \end{bmatrix}=\begin{bmatrix}
         x_{11}x_{22}-x_{12}x_{21} & x_{11}x_{32}-x_{12}x_{31} & x_{21}x_{32}-x_{22}x_{31} \\
         x_{11}x_{23}-x_{13}x_{21} & x_{11}x_{33}-x_{13}x_{31} & x_{21}x_{33}-x_{23}x_{31} \\
         x_{12}x_{23}-x_{13}x_{22} & x_{12}x_{33}-x_{13}x_{32} & x_{22}x_{33}-x_{23}x_{32}
     \end{bmatrix}.\]
     Hence, picking $a_i,b_i,c_i$ as given in $Y$ we get $\varphi\in \Aut(H)$ determined by the $g_i$ in (\ref{eq:x'}) that gives a Frobenius group $H\rtimes_{\varphi}\Z_{|\varphi|}$ with a Camina group kernel of order $p^6$ and nilpotency class $2$. Thus, $T(G,\C[G]^H)$ is AC.}
\end{Examp}

\begin{Examp}
    An example with a Frobenius kernel of nilpotency class $3$ can be constructed as follows. Let $H=G_{(11^7,750208)}$ (using the Magma index). This group is generated by $a,b,c,d,e,f,g$ of order $p$ with relations 
    \[b^a=be,\ c^a=cf,\ c^b=cg,\ d^b=dfg^{10},\ d^c=de^2g^{10},\ e^a=eg,\ f^d=fg^9.\] 
    Also $g$ is central and all other commutators of generators are trivial. This is a Camina $11-$group with nilpotency class $3$. Define the automorphism $\varphi\colon H\rightarrow H$ by 
    \[\varphi(a)=a^4e^7f^7g^7,\ \varphi(b)=b^5e^6fg^8,\ \varphi(c)=c^5e^2f^6g,\]
    \[\varphi(d)=d^4e^8g^4,\ \varphi(e)=e^9g^{10},\ \varphi(f)=f^9g^8,\ \varphi(g)=g^3.\]
    This automorphism can be shown to have no nontrivial fixed points and so gives a Frobenius action. It has order $5$ giving a Frobenius group $H\rtimes_{\varphi} \Z_5$ with a Frobenius kernel that is a Camina $p-$group with nilpotency class $3$ and cyclic complement. Thus, $T(G,\C[G]^H)$ is AC.

    Using Magma one finds that $G_{(11^7,750208)}$ is the only group of order $11^7$ with an automorphism that is fixed point free. So, the group $G$ constructed above is the only Frobenius group with Frobenius kernel of order $11^7$.
\end{Examp}

We conclude this section by finding $\dim T(G,\C[G]^H)$ in the nonabelian case of Theorem \ref{thm:sgpclassify}.

\begin{Prop}\label{prop:sgpdim}
    Let $G$ be a Frobenius group with Frobenius kernel $H$ and cyclic complement of order $k$. Assume that $T(H)$ is AC and $H$ has $m$ conjugacy classes. Then 
    \[\dim T(G,\C[G]^H)=\dim T(H)+(k-1)(k-2+3m).\]
\end{Prop}

\begin{proof}
    By Theorem \ref{thm:sgpclassify}, $T(G,\C[G]^H)$ is AC and $H=G'$. By Theorem \ref{thm:tanaka}, $T(G,\C[G]^H)$ is triply regular, so $ T(G,\C[G]^H)=T_0(G,\C[G]^H)$ and $\dim T_0(G,\C[G]^H)=|\{(i,j,k)\colon p_{ij}^k\neq 0\}|$. We count the number of nonzero $p_{ij}^k$. Let $P_0=g_0^H,\cdots, P_d=g_d^H$ be the principal classes of $\C[G]^H$. The principal sets of $\C[H]^H$ are the principal sets of $\C[G]^H$ contained in $H$. We consider cases.
    
    {\bf Case 1: }$P_i,P_j\subseteq H$. Then $P_i=h_i^H$ and $P_j=h_j^H$ for some $h_i,h_j\in H$. Since $P_i$ and $P_j$ are both principal sets of $\C[H]^H$,  $p_{ij}^k\neq 0$ in $\C[H]^H$ if and only if $p_{ij}^k\neq 0$ in $\C[H]$. Therefore the number of triples $(i,j,k)$ such that $p_{ij}^k\neq 0$ is equal to $\dim T_0(H)$. However, $T(H)$ is AC, so by Theorem \ref{thm:tanaka}, $T(H)=T_0(H)$. Therefore, we have $\dim T(H)$ nonzero triples $(i,j,k)$ in this case.
    
    {\bf Case 2: }$P_i\not\subseteq H$ and $P_j\subseteq H$. Since $P_i\not\subseteq H$, $g_i\not\in H$, so by Proposition \ref{prop:frobcon}, $g_i^H=g_iH$. Therefore, $P_iP_j=g_iHP_j=g_iH$ as $P_j\subseteq H$. Therefore, $p_{ij}^k\neq 0$ if and only if $k=i$ in this case. Thus, for each choice of $P_j\subseteq H$ we have one choice of $k$. There are $m$ principal classes $P_j\subseteq H$. Therefore, for each choice of $i$ there are $m$ different $p_{ij}^k\neq 0$. As $P_i=g_i^H=g_iH$ for all $P_i\not\subseteq H$, the number of distinct choices for $P_i$ is $k-1$, since $|G/H|=k$. Therefore we have $(k-1)m$ different $p_{ij}^k$ in this case.
    
    {\bf Case 3: } $P_i\subseteq H$ and $P_j\not\subseteq H$. As in Case $2$, we get $(k-1)m$ nonzero triples.
    
    {\bf Case 4: } $P_i,P_j\not\subseteq H$.  Since $P_i,P_j\not\subseteq H$, we have $g_i,g_j\not\in H$. Then by Proposition \ref{prop:frobcon}, $g_i^H=g_iH$ and $g_j^H=g_jH$. Fix $P_i$ and consider two subcases.\\
    {\bf Subcase 4.1: }$P_j=P_i^*$. Then $P_j^*=(g_i^{-1})H$ and so $P_iP_j=g_iH(g_i^{-1}H)=H$. Then $p_{ij}^k\neq 0$ for every $P_k\subseteq H$. This is the number of conjugacy classes of $H$. Hence, we have $m$ triples $(i,j,k)$ with $p_{ij}^k\neq 0$ in this case.\\
    {\bf Subcase 4.2: }$P_j\neq P_i^*$. So $g_jH\neq g_i^{-1}H$. Then $P_iP_j=(g_ig_j)H$. As $g_jH\neq g_i^{-1}H$, $g_ig_j\not\in H$. By Proposition \ref{prop:frobcon}, $g_ig_jH=(g_ig_j)^H=P_k$ for some $P_k\not\subseteq H$. There is a unique choice of $k$ such that $p_{ij}^k\neq 0$ in this case. Then for each choice of $P_j\neq P_i^*$, we have one triple $(i,j,k)$ such that $p_{ij}^k\neq 0$. There are a total of $k-1$ choices for $P_j$, one of which has $P_j=P_i^*$. Thus, we have $k-2$ nonzero $(i,j,k)$ in this case.\\
    From these two subcases we see that for each choice of $P_i$ there are $k-2+m$ nonzero triples $(i,j,k)$. There are a total of $k-1$ choices for $P_i$. We thus have $(k-1)(k-2+m)$ different nonzero triples $(i,j,k)$ in this case. Summing gives:
    \[\mathsmaller \dim T(G,\C[G]^H)=\dim T(H)+(k-1)m+(k-1)m+(k-1)(k-2+m)=\dim T(H)+(k-1)(k-2+3m). \qedhere\]
\end{proof}

We note that $\dim T(H)$ in the previous theorem is computed in \cite{Bastian2} for each possible $H$.

\section{Wreath Products}
Suppose $(G,H)$ is a SGP. We let $\mathcal{S}(G,H)$ denote the association scheme corresponding to the Schur ring $\C[G]^H$. For any Frobenius group $G$ with abelian complement, we have by Proposition \ref{prop:frobcon}, that for all $x\in G\setminus G'$, $x^{G'}=xG'$ and $G'$ is the Frobenius kernel of $G$. By Corollary \ref{cor:frobsgp}, $(G,G')$ is a SGP. Let $\Lambda$ be any ordering of $G'$ and let $g_1=e,g_2,\cdots, g_m$ be representatives for $G/G'$. We order $G$ as: $\Lambda, g_2\Lambda,\cdots, g_m\Lambda$. We start by discussing the form of the adjacency matrices in $\mathcal{S}(G,H)$.

\begin{Lem}\label{lem:sgpouter}
    Let $G$ be a Frobenius group with abelian complement. Let $P_k$ be a principal class of $\C[G]^{G'}$ such that $P_k\not\subseteq G'$. Then for any principal classes $P_i,P_j$ of $\C[G]^{G'}$ the $P_i,P_j$ block of $A_k$ is all $0$'s if $P_j\not\subseteq P_iP_k$ and all $1$'s if $P_j\subseteq P_iP_k$.
\end{Lem}

\begin{proof}
    As $P_k\not\subseteq G'$ and $G$ is a Frobenius group with cyclic complement, $P_k=gG'$ for some $g\in G\setminus G'$ by Proposition \ref{prop:frobcon}. Now suppose $x\in P_i$ and $y\in P_j$. We have $(A_k)_{xy}=1$ if and only if $yx^{-1}\in gG'$, if and only if $y\in gG'x=gxG'$, if and only if $P_j\subseteq gxG'$. Now either $P_i\subseteq G'$ or $P_i=hG'$ for some $h\in G\setminus G'$ by Proposition \ref{prop:frobcon}. If $P_i\subseteq G'$, then $x\in G'$ and $gxG'=gG'=P_iP_k$. If $P_i=hG'$ for some $h\in G\setminus G'$, then $P_i=xG'$ as well. Then $gxG'=P_kP_i=P_iP_k$. Either way $P_iP_k=gxG'$, and so for $x\in P_i$ and $y\in P_j$, $(A_k)_{xy}=1$ if and only if $P_j\subseteq P_iP_k$.
 
    Now suppose the $P_i,P_j$ block of $A_k$ has a nonzero entry, say $(A_k)_{ab}=1$. This implies $P_j\subseteq P_iP_k$, and so for all $x\in P_i$ and $y\in P_j$ that $(A_k)_{xy}=1$. Thus, the $P_i,P_j$ block of $A_k$ is all $1$'s.
\end{proof}

\begin{Lem}\label{lem:sgpcenteral}
    Let $G$ be a Frobenius group with abelian complement. Let $P_k=g^{G'}$ be a principal class of $\C[G]^{G'}$ such that $P_k\subseteq G'$. Let $x^{G'}$ and $y^{G'}$ be two principal classes of $\C[G]^{G'}$. We then have
    \begin{enumerate}[leftmargin=*]
        \item If $x,y\in G'$, then the $x^{G'},y^{G'}$ block of $A_k$ is equal to the $x^{G'},y^{G'}$ block of the adjacency matrix of $g^{G'}$ for the association scheme $\mathcal{G}(G')$.
        \item If $x^H\neq y^H$ with either $x\not\in G'$ or $y\not\in G'$, then the $P_i,P_j$ block of $A_k$ is all $0$.
        \item If $x^H\not\subseteq Z(G)$, then the $x^H,x^H$ block of $A_k$ is equal to the adjacency matrix of $g^{G'}$ for $\mathcal{G}(G')$.
    \end{enumerate}
\end{Lem}

\begin{proof}
    As $P_k\subseteq G'$, $g\in G'$. First let $x,y\subseteq G'$. Let $(a,b)$ be any entry of the $x^{G'},y^{G'}$ block of $A_k$. We have $a\in x^{G'}\subseteq G'$ and $b\in y^{G'}\subseteq G'$. Notice $(A_k)_{ab}=1$ if and only if $ba^{-1}\in P_k=gG'$ if and only if the $(a,b)$ entry of the adjacency matrix of $g^{G'}$ for $\mathcal{G}(G')$ is nonzero. Furthermore we note as $x^{G'},y^{G'}\subseteq G'$ we have a $x^{G'},y^{G'}$ block of the adjacency matrix for $g^{G'}$ in $\mathcal{G}(G')$. So the $x^{G'},y^{G'}$ block of $A_k$ is equal to the $x^{G'},y^{G'}$ block of the adjacency matrix of $g^{G'}$ for $\mathcal{G}(G')$.  

    Next consider when $x^{G'}\neq y^{G'}$ with either $x\not\in G'$ or $y\not\in G'$. Without loss of generality suppose $x\not\in G'$. By Proposition \ref{prop:frobcon}, $x^{G'}=xG'$. Let $a\in x^{G'}$ and $b\in y^{G'}$. If $(A_k)_{ab}=1$, then $ba^{-1}\in P_k\subseteq G'$. We note as $a\in xG'$, that $a^{-1}\in x^{-1}G'$. If $y\in G'$, then $y^{G'}\subseteq G'$ and $ba^{-1}\in x^{-1}G'$. However, $x^{-1}G'\cap G'=\emptyset$, which contradicts $ba^{-1}\in x^{-1}G'\cap G'$. If $y\not\in G'$, then by Proposition \ref{prop:frobcon}, $y^{G'}=yG'$. Since $x^{G'}\neq y^{G'}$, we have $xG'\neq yG'$, so $yx^{-1}G'\neq G'$. As $b\in yG'$ and $a^{-1}\in x^{-1}G'$, $ba^{-1}\in yx^{-1}G'$. However, $yx^{-1}G'\cap G'=\emptyset$ which contradicts $ba^{-1}\in yx^{-1}G'\cap G'$. In either case we have a contradiction, therefore $(A_k)_{ab}\neq 1$. Thus for $a\in x^{G'}$, $b\in y^{G'}$ we have $(A_k)_{ab}=0$, and the $x^{G'},y^{G'}$ block of $A_k$ is $0$.

    Lastly consider when $x^{G'}=y^{G'}$ with $x\not\in G'$. Since $x\not\in G'$ by Proposition \ref{prop:frobcon}, $x^{G'}=xG'$. Let $a,b\in x^{G'}=xG'$. Say $a=xh_1$ and $b=xh_2$ for $h_1,h_2\in G'$. Then $(A_k)_{ab}=1$ if and only if $ba^{-1}\in g^{G'}$ if and only if $xh_2(xh_1)^{-1}=xh_2h_1^{-1}x^{-1}\in gG'$ if $h_2h_1^{-1}\in x^{-1}gG'x=gG'$. Hence, the $(xh_1,xh_2)$ entry of $A_k$ is nonzero if and only if the $(h_1,h_2)$ entry of the adjacency matrix for $gG'$ in $\mathcal{G}(G')$ is nonzero. Under our ordering of the elements of $G$, notice that the $(xh_1,xh_2)$ entry of the $x^{G'},x^{G'}$ block of $A_k$ is in the same row and column of the $x^{G'},x^{G'}$ block of $A_k$ as the $(h_1,h_2)$ entry of the adjacency matrix for $gG'$ in $\mathcal{G}(G')$. Therefore, the $x^{G'},x^{G'}$ block of $A_k$ is equal to the adjacency matrix of $g^{G'}$ for the association scheme $\mathcal{G}(G')$.   
\end{proof}

\begin{Lem}\label{lem:sgpclass}
 Let $G$ be a Frobenius group with abelian complement. Suppose $|G|=n$ and $|G'|=m$. If $P_k=g^{G'}$ is a principal class of $\C[G]^{G'}$ such that $g\in G'$, then $A_k=I_{n/m}\otimes B$ where $B$ is the adjacency matrix of $g^{G'}$ for the association scheme $\mathcal{G}(G')$.  
\end{Lem}

\begin{proof}
     We consider the blocks of $A_k$ as cosets of $G'$ using the ordering described above. By Proposition \ref{prop:frobcon}, $P_k=g^{G'}=gG'$. The $xG',yG'$ block of $A_k$ is the same as the $x^{G'},y^{G'}$ block for any $x,y\not\in G'$ by Proposition \ref{prop:frobcon}. So in this situation the blocks are the same as the principal class blocks. The only difference is that classes in $G'$ are now being considered as a single block. If $xG'\neq yG'$ for any $x,y\in G$ either $x\not\in G'$ or $y\not\in G'$ in this case. Then by Lemma \ref{lem:sgpcenteral}, the $xG',yG'$ block of $A_k$ is all $0$. Hence, only the diagonal blocks are nonzero. Now consider the $xG',xG'$ block of $A_k$. If $xG'\neq G'$, then this is the same as the $x^{G'},x^{G'}$ block of $A_k$. By Lemma \ref{lem:sgpcenteral}, this block equals $B$. Next we consider the $G',G'$ block of $A_k$. By Lemma \ref{lem:sgpcenteral}, the $x^{G'},y^{G'}$ block of $A_k$ equals the $x^{G'},y^{G'}$ block of $B$ for all $x,y\in G'$. Therefore, the $G',G'$ block of $A_k$ is just $B$.

    Thus $A_k$ is a block diagonal matrix in which each diagonal block is $B$. So, $A_k=I_{n/m}\otimes B$. 
\end{proof}

\begin{Lem}\label{lem:sgpclass'}
 Let $G$ be a Frobenius group with abelian complement. Suppose $|G|=n$ and $|G'|=m$. If $P_k$ is a principal class of $\C[G]^{G'}$ such that $P_k\not\subseteq G'$, then $A_k=D\otimes J_{m}$ where $D$ is the adjacency matrix of the conjugacy class $\{gG'\}$ in $\mathcal{G}(G/G')$.
\end{Lem}

\begin{proof}
     Let $P_k=g^{G'}$ and notice as $g\not\in G'$ that $g^{G'}=gG'$ by Proposition \ref{prop:frobcon}. Let us consider the blocks of $A_k$ as cosets of $G'$ using the ordering described above. In doing so, the $xG',yG'$ block of $A_k$ is the same as the $x^{G'},y^{G'}$ block for $x,y\not\in G'$. So in this regard the blocks are the same as the principal class blocks. The only difference is that classes in $G'$ are now being considered as a single block. From Lemma \ref{lem:sgpouter}, the $P_i,P_j$ block of $A_k$ is all $1$'s if $P_j\subset P_iP_k$ and is all $0$'s otherwise. 
     
     For any $P_i\subset G'$, the $P_i,P_j$ block of $A_k$ is nonzero if and only if $P_j\subset P_iP_k=gG'=P_k$ if and only if $P_j=P_k$. Hence, the $P_i,P_j$ block is nonzero if and only if $P_j=gG'$. This is true for all $P_i\subset G'$. Then taking the union over all of the $P_i\subset G'$, the $G',gG'$ block of $A_k$ is all $1$'s and the $G',yG'$ block for $yG'\neq gG'$ is all $0$'s.

     Next consider when $P_j\subset G'$, $P_j=x^{G'}$, $x\in G'$. The $P_i,P_j$ block of $A_k$ is nonzero if and only if $P_j\subset P_iP_k$ if and only if $x\in P_iP_k$. Now if $P_i\subset G'$ we have $P_iP_k=gG'=P_k$. Since $x\in G'$, we have $x\not\in gG'$. So for all $P_j\subset G'$, the $P_i,P_j$ block of $A_k$ is zero. When $P_i=yG'$ we have $x\in P_iP_k=(yG')(gG')=ygG'$ if and only if $ygG'=G'$, if and only if $yG'=g^{-1}G'$. Hence, in this case the $P_i,P_j$ block of $A_k$ is all $1$'s if and only if $P_i=g^{-1}G'$. This is true for all $P_j\subset G'$. Then taking the union over all of the $P_j\subset G'$, the $g^{-1}G',G'$ block of $A_k$ is all $1$'s and the $yG',G'$ block for $yG'\neq g^{-1}G'$ is all $0$'s. 

     We now consider the case where both $P_i,P_j\not\subset G'$. Say $P_i=x^{G'}$ and $P_j=y^{G'}$. By Proposition \ref{prop:frobcon}, $P_i=xG'$ and $P_j=yG'$. From Lemma \ref{lem:sgpouter}, the $xG',yG'$ block of $A_k$ is all $1$'s if $P_j\subset P_iP_k$ and is all $0$'s otherwise. Then the block is all $1$'s if $yG'\subset (xG')(gG')=xgG'$ if and only if $yG'=xgG'=gxG'$, since $G/G'$ is abelian. Thus the $xG',gxG'$ block of $A_k$ is all $1$'s for all $x\not\in G'$ and the $xG',yG'$ block of $A_k$ is all $0$'s if $yG'\neq gxG'$. 

     From these three cases we see that the $xG',gxG'$ block of $A_k$ is all $1$'s for all $x\in G$ and is all $0$'s otherwise. Then $A_k=D\otimes J_m$ where $D$ is the $(n/m)\times (n/m)$ matrix that is nonzero in the $(x,gx)$ entry for all $x\in G$ and is $0$ otherwise. Notice in $\mathcal{G}(G/G')$ the adjacency matrix $A$ for the class $\{gG'\}$ is nonzero in the $(xG',yG')$ entry if and only if $yx^{-1}G'=gG'$ if and only if $y=gxG'$. Hence, $A_{xG',yG'}=1$ if and only if $y=gxG'$ if and only if $D$ is nonzero in the (x,gx) entry. Hence, $D$ is the adjacency matrix for $\{gG'\}$ when considered as a conjugacy class in $G/G'$.
\end{proof}

\begin{Cor}\label{cor:sgpwreath}
     Let $G$ be a Frobenius group with abelian complement. Suppose $|G|=n$ and $|G'|=m$. Then $\mathcal{S}(G,G')=\mathcal{G}(G')\wr \mathcal{G}(G/G')$.
\end{Cor}

\begin{proof}
    Let $A_0,A_1,\cdots, A_{r}$ be the adjacency matrices of $\mathcal{G}(G')$ and $B_0,B_1,\cdots, B_{n/m}$ be the adjacency matrices of $\mathcal{G}(G/G')$. The adjacency matrices of $\mathcal{G}(G')\wr \mathcal{G}(G/G')$ are just $D_0=I_{p^{n/m}}\otimes I_{m}=I_{n}$, $D_k=I_{{n/m}}\otimes A_k$ where $1\leq k\leq r$, and $D_{r+k}=B_k\otimes J_{m}$ for $1\leq k\leq n/m$.  From Lemma \ref{lem:sgpclass}, all adjacency matrices $M_k$ of $\mathcal{S}(G,G')$ for some class $P_k=g^{G'}\subset G'$ are of the form $M_k=I_{p^{n/m}}\otimes A_k$ where $A_k$ is the adjacency matrix for $g^{G'}$ in $G'$. Ranging over all $P_k\subset G'$ we have $D_i=M_i$ for all $0\leq i\leq r$. From Lemma \ref{lem:sgpclass'}, all adjacency matrices $M_k$ of $\mathcal{S}(G,G')$ for some class $P_k=g^{G'}$ outside of $G'$ are of the form $M_k=B_\ell\otimes J_{m}$ where $B_\ell$ is the adjacency matrix for $\{gG'\}$ in $G/Z(G)$. We then have a one-to-one correspondence between the adjacency matrices $M_k$ of $\mathcal{S}(G,G')$ for $P_k\not\subseteq G'$ and the adjacency matrices of the form $D_{r+k}=B_k\otimes J_{m}$ for $1\leq k\leq n/m$ in $\mathcal{G}(G')\wr \mathcal{G}(G/G')$.  

    We then see that $\mathcal{S}(G,G')$ and $\mathcal{G}(G')\wr \mathcal{G}(G/G')$ have the exact same adjacency matrices. Therefore, $\mathcal{S}(G,G')=\mathcal{G}(G')\wr \mathcal{G}(G/G')$.    
\end{proof}

\printbibliography

\end{document}